\documentclass[11pt,reqno]{amsart}

\usepackage{a4wide}
\usepackage{amsmath,amssymb} 
\usepackage{bbm}
\usepackage{color}
\usepackage{tikz}
\usepackage{tkz-graph}
\tikzset{
  LabelStyle/.style = { rectangle, rounded corners, draw,
                        minimum width = 1em, 
                        font =  },
  VertexStyle/.append style = { inner sep=3pt,
                                font = \large},
  EdgeStyle/.append style = {->, bend left} }

\tikzset{
    ultra thick/.style={line width=5.0pt}
}

\theoremstyle{plain}
\newtheorem{theorem}{Theorem}
\newtheorem{prop}[theorem]{Proposition}
\newtheorem{lemma}[theorem]{Lemma}
\newtheorem{coro}[theorem]{Corollary}

\theoremstyle{definition}
\newtheorem{definition}[theorem]{Definition}
\newtheorem{remark}[theorem]{Remark}
\newtheorem{question}[theorem]{Question}
\newtheorem{example}[theorem]{Example}

\newcommand{\Z}{{\mathbb Z}}
\newcommand{\Q}{{\mathbb Q}}
\newcommand{\R}{{\mathbb R}}
\newcommand{\N}{{\mathbb N}}

\newcommand{\mc}{\mathcal}

\newcommand{\dsub}{\varphi}

\begin{document}

\title{
Shifts of finite type and random substitutions
}

\author{Philipp Gohlke}
\author{Dan Rust}
\author{Timo Spindeler}
\date{\today}
\address{Fakult\"at f\"ur Mathematik, Universit\"at Bielefeld, \newline
\hspace*{\parindent}Postfach 100131, 33501 Bielefeld, Germany}
\email{\{pgohlke,drust,tspindel\}@math.uni-bielefeld.de}

\begin{abstract}
We prove that every topologically transitive shift of finite type in one dimension is topologically conjugate to a subshift arising from a primitive random substitution on a finite alphabet.
As a result, we show that the set of values of topological entropy which can be attained by random substitution subshifts contains all Perron numbers and so is dense in the positive real numbers.
We also provide an independent proof of this density statement using elementary methods.
\end{abstract}



\maketitle

\section{Introduction}\label{SEC:intro}
One of the first appearances of random substitutions in the literature is the work of Godr\`{e}che and Luck \cite{GL:random} in the late 80s where they introduced the \emph{random Fibonacci substitution} as a toy model for a one-dimensional quasicrystal with positive topological entropy but still retaining long range order---a quasicrystal lying in the intermediate regime between totally ordered and totally random.
A random substitution is a generalisation of the classical notion of a substitution on a finite alphabet \cite{F:book} whereby, instead of letters having a pre-determined image under the substitution, the random substitution assigns (with a given probability distribution) a finite (or possibly infinite) set of possible images to each letter.
The random substitution is applied independently to each letter in any finite or infinite word.
Since their work, there have been a handful of advances in the study of random substitutions, as well as several instances of independent discovery such as the essentially identical notions of Dekking and Meester's \emph{multi-valued substitutions} \cite{DM:multi} or the \emph{$0L$-systems} of formal language theory \cite{RS:0L}.

There has been a recent attempt to bring the theory of random substitutions squarely into the realm of symbolic dynamics \cite{RS:random} so that the dynamical and metric formalisms available to the study of subshifts over finite alphabets can be utilised.
There, the subshifts associated to random substitutions were called \emph{RS-subshifts} and we laid the groundwork for the topological and dynamical theory of RS-subshifts, as well as beginning to understand the ergodic properties of RS-subshifts in terms of the assigned probabilities of image words under substitution.
Topological entropy is of particular interest when studying RS-subshifts \cite{bible, BSS:random, M:moll-thesis, M:diffraction-noble, N:fibonacci-entropy} and remains relatively unexplored, except in cases of specific examples or families of examples.

We continue that work here by highlighting the strong links between RS-subshifts and one-dimensional shifts of finite type (SFTs) \cite{LM:introduction-to-symbolic}.
We should mention that there are links that have been previously explored between SFTs in dimensions greater than one and classical deterministic substitutions in Euclidean space.
In particular, celebrated work of Goodman-Strauss \cite{GS:matching-rules} has shown that every substitution tiling of $\mathbb{R}^n$ (under mild restrictions) admits a finite set of marked tiles with local matching rules (a generalisation of higher dimensional SFTs) such that those tiles can only be arranged into tilings with the same local structure as the substitution tiling. Fernique and Ollinger have also provided a generalisation of this result to the combinatorial setting \cite{FO:sofic}.
It should also be noted that, in dimension one, no aperiodic deterministic substitution subshifts can be topologically conjugate to an SFT as one-dimensional SFTs necessarily contain periodic points---hence Goodman-Strauss' result cannot apply in dimension one.

As their study is still in its infancy, there remain many open problems about RS-subshifts.
One of a general class of open questions relating to the above is the following: When can one recognise a given subshift as an RS-subshift (possibly up to conjugacy)? The main result of this article addresses this question for SFTs.
We are able to show that ever topologically transitive shift of finite type is topologically conjugate to an RS-subshift.
Our proof relies on the interpretation of SFTs as vertex shifts on finite directed graphs and an elementary decomposition result of cycles in such graphs into simple cycles.
An immediate corollary of this result is that the logarithm of every Perron number (an algebraic number which dominates all it algebraic conjugates in absolute value) appears as a value of topological entropy for some RS-subshift.
As a consequence, the set of possible values for topological entropy of RS-subshifts is a dense subset of the positive reals---this weaker statement warrants an independent proof using elementary techniques, and this is provided in Section \ref{SEC:entropy}.

The paper is structured as follows:
In Section \ref{SEC:cycles}, we review some basic terminology from the theory of directed graphs and prove a key decomposition result for directed cycles which will be used later in the proof of the main theorem.
In Section \ref{SEC:sft}, we review basic results relating to one-dimensional shifts of finite type.
Section \ref{SEC:rs-subshift} gives a brief introduction to subshifts associated with random substitutions and outlines some of the key results from previous work of two of the authors \cite{RS:random}.
Section \ref{SEC:main} is devoted to the proof of our main theorem wherein we show that every topologically transitive SFT is topologically conjugate to an RS-subshift.
We also provide examples showing that our proof is constructive and can be applied in practice.
As an aside, we also briefly explain and illustrate via an example how these techniques can easily be modified to SFTs represented as edge shifts, rather than vertex shifts.
We conclude with Section \ref{SEC:entropy}, wherein we provide a series of results concerning the topological entropy of RS-subshifts.
Methods applied in this section appear to be widely applicable and ubiquitously useful whenever one wishes to bound (or explicitly calculate) the entropy of an RS-subshift, and which are especially convenient when trying to `cook up' examples of RS subshifts with entropy satisfying some bounding or valuation criteria.

\section{Cycles in directed graphs}\label{SEC:cycles}
A \emph{finite directed graph} or just \emph{graph} $G$ is a pair $G = (V,E)$ where $V$ is a finite set and $E$ is a subset of $V^2$.
The \emph{source map} $s \colon E \to V$ and the \emph{target map} $t \colon E \to V$ are given by projecting to the first and second components respectively of $V^2$.
That is, $s(v_1,v_2) = v_1$ and $t(v_1,v_2) = v_2$.

Our graphs will always be finite, and essential.
That is, for every vertex $v$, we have at least one edge $e$ with $s(e) = v$ and at least one edge $f$ with $t(f) = v$.
Moreover, unless explicitly stated, we will not consider multiple edges between a pair of vertices.
\begin{definition}
Let $G = (V, E)$ be a graph.
A \emph{path} $\gamma$ of length $\ell$ in $G$ is a finite non-empty ordered tuple of edges $\gamma = (e_1, \cdots, e_{\ell})$ such that $t(e_i) = s(e_{i+1})$ for all $1 \leq i \leq \ell - 1$.
The \emph{source} $s(\gamma)$ of $\gamma$ is $s(\gamma) = s(e_1)$ and the \emph{target} $t(\gamma)$ of $\gamma$ is $t(\gamma) = t(e_{\ell})$.
If $s(\gamma) = t(\gamma)$, we call $\gamma$ a \emph{cycle}.
The \emph{root} of a cycle $\gamma$ is the vertex $s(\gamma) = t(\gamma)$.
A path or cycle $\gamma = (e_1, \cdots, e_\ell)$ is called \emph{simple} if $s(e_i) = s(e_j) \implies i = j$.
That is, $\gamma$ is a path or cycle which traverses each vertex at most once.
We say that $G$ is \emph{strongly connected} if for every $v, w \in V$, there exists a path $\gamma$ whose source and target are $s(\gamma) = v$ and $t(\gamma) = w$ respectively.
\end{definition}
\begin{remark}
Our definitions are in line with those in the book of Lind and Marcus \cite{LM:introduction-to-symbolic}. However, note that the definition of path given above differs from that appearing in some other texts such as Diestel's \cite{D:book-diestel}. In particular, our definition of path allows for the traversing of the same edge more than once.
\end{remark}
\begin{lemma}\label{LEM:cycles}
Let $G$ be a graph. If $G$ is strongly connected, then for every $v,w \in V$, there exists a cycle $\gamma = (e_1, \ldots, e_\ell)$ rooted at $v$ such that $s(e_i) = w$ for some edge $e_i$ appearing in $\gamma$.
\end{lemma}
\begin{proof}
The graph $G$ is strongly connected, so let $\delta = (e_1, \ldots, e_i)$ be a path such that $s(\delta) = v$ and $t(\delta) = w$, and let $\eta = (f_1, \ldots, f_j)$ be a path such that $s(\eta) = w$ and $t(\eta) = v$. Then the concatenation of $\delta$ and $\eta$ given by $\gamma = \delta \# \eta = (e_1, \ldots, e_i, f_1, \ldots, f_j)$ is a cycle rooted at $v$ because $t(e_i) = w = s(f_1)$ and $s(\gamma) = s(e_1) = v = t(f_j) = t(\gamma)$.
\end{proof}
\begin{lemma}\label{LEM:simple-cycles}
Let $G = (V,E)$ be a graph and let $\gamma = (e_1, \ldots, e_\ell)$ be a cycle in $G$.
Either $\gamma$ is a simple cycle, or there exists a simple subcycle $\gamma_0 = (e_i, \ldots, e_j)$ of $\gamma$ such that $\tilde{\gamma}_1 = (e_1, \ldots, e_{i-1}, e_{j+1}, \ldots e_{\ell})$ is a cycle.
\end{lemma}
\begin{proof}
Suppose $\gamma$ is not simple. Then there exist $1 \leq i < j \leq \ell$ such that $s(e_i) = s(e_j)$.
The subpath $\gamma' = (e_i, \ldots, e_{j-1})$ is a cycle because $t(e_{j-1}) = s(e_j) = s(e_i)$.
The cycle $\gamma'$ has length $\ell'$ strictly smaller than $\ell$ by the bounds imposed on $i$ and $j$.
If $\gamma'$ is simple then we are done.
If not, then we repeat the process of finding a subcycle of $\gamma'$ which must also be a subcycle of $\gamma$.
As every cycle of length $1$ is simple, it follows that this inductive process must eventually find a simple subcycle $\gamma_0$ of $\gamma$.

Now consider the removal of a subcycle $\gamma_0 = (e_i, \ldots, e_j)$ from a cycle $\gamma  = (e_1, \ldots, e_\ell)$ which leaves a sequence of edges $\tilde{\gamma}_1 = (e_1, \ldots, e_{i-1}, e_{j+1}, \ldots e_{\ell})$.
From the cycle relations, we have $t(e_{i-1}) = s(e_i) = t(e_j) = s(e_{j+1})$ and so $\tilde{\gamma}_1$ is also a cycle.
\end{proof}
This gives us a decomposition result for writing cycles in a directed graph as `nested insertions' of simple cycles.
\begin{prop}[Cycle decomposition]\label{PROP:cycle-decomp}
For every cycle $\gamma$, there exists a finite sequence of simple cycles $\gamma_0, \ldots, \gamma_k$ and a sequence of cycles $\tilde{\gamma}_0, \ldots, \tilde{\gamma}_{k+1}$ such that $\tilde{\gamma}_0 = \gamma$, $\tilde{\gamma}_{k+1}$ is simple, $\gamma_i$ is a subcycle of $\tilde{\gamma}_i$ and $\tilde{\gamma}_{i+1}$ is given by removing $\gamma_i$ from $\tilde{\gamma}_i$.
\end{prop}
\begin{proof}
This is an easy repeated application of Lemma \ref{LEM:simple-cycles} and noting that $\tilde{\gamma}_{i+1}$ is strictly shorter as a cycle than $\tilde{\gamma}_i$.
\end{proof}
\begin{remark}
Note that the cycle decomposition given by Proposition \ref{PROP:cycle-decomp} is not unique.
Take as a basic example any concatenation of two simple cycles rooted at the same vertex.
The process can be made canonical however (for instance in order to be implemented in a computer) by scanning through a cycle from left to right until the first simple cycle is found and removed, and then repeating the process.
\end{remark}

\section{Shifts of finite type}\label{SEC:sft}

\subsection{Introduction to SFTs}
A good introduction to shifts of finite type and symbolic dynamics in general is the book of Lind and Marcus \cite{LM:introduction-to-symbolic}.
There are several equivalent definitions of shifts of finite type up to topological conjugacy.
We will freely pass between these treatments.

Let $\mc A$ be a finite alphabet, $\mc A^{+}$ the set of words in $\mc A$ with finite length and $\mc A^{\ast} = \mc A^{+} \cup \{\varepsilon\}$, where $\varepsilon$ denotes the empty word.
We write the usual concatenation of words as $u v$ for $u,v \in \mc A^{+}$, and $v \triangleleft u$ if $v$ is a subword of $u$.
Under word concatenation, $\mc A^\ast$ forms a free monoid generated by $\mc A$ with identity given by the empty word $\varepsilon$.
All subsets of $\mc A^\Z$ for a finite alphabet $\mc A$ will automatically adopt the subspace topology, where $\mc A$ is a discrete space and $\mc A^\Z$ is endowed with the product topology.

Let $A$ be an $n \times n$ matrix with entries in $\{0,1\}$.
The \emph{shift of finite type} $X_A$ associated to the matrix $A$ is the space
\[
X_A = \{\cdots x_{-1} x_0 x_1 \cdots = x \in \{0, \ldots, n-1\}^{\Z} \mid A_{x_{i}, x_{i+1}} = 1\}.
\]

Let $\mc F$ be a finite subset of $\mc A^+$ which we call a set of \emph{forbidden words}.
The \emph{shift of finite type} $X_{\mc F}$ associated to $\mc F$ is the space
\[
X_{\mc F} = \{ x \in \mc A^\Z \mid u \triangleleft x \implies u \notin \mc F\}.
\]
Note that, up to topological conjugacy of $X_{\mc F}$, we may assume that all words $u$ in $\mc F$ have length $|u| = 2$  by taking \emph{higher powers} \cite[Prop. 2.3.9]{LM:introduction-to-symbolic}.

Let $G = (V,E)$ be a directed graph.
The \emph{vertex shift} $\widehat{X}_G$ of $G$ is the space
\[
\widehat{X}_G = \{x \in V^\Z \mid (x_i, x_{i+1}) \in E\}.
\]
Vertex shifts are SFTs.
There is a further notion of \emph{edge shifts} on directed graphs.
We will revisit edge shifts later in Section \ref{SEC:main}.

\begin{definition}
Let $A$ be a non-negative integer matrix.
We call $A$ \emph{irreducible} if, for every $i,j$, there exists an integer $k \geq 1$ such that $(A^k)_{ij} \geq 1$.
We call $A$ \emph{primitive} if there exists an integer $k \geq 1$ such that $(A^k)_{ij} \geq 1$ for all $i,j$.
\end{definition}
It is also well known that topologically transitive SFTs have a dense subset of periodic points.
The following result is also well known and will be used in the proof of our main result in Section \ref{SEC:main}.
\begin{prop}
Let $A$ be a $0$-$1$ matrix and let $G_A$ be a graph whose adjacency matrix is $A$.
The following are equivalent.
\begin{itemize}
\item $A$ is irreducible.
\item $G_A$ is strongly connected.
\item $X_A$ is topologically transitive.
\end{itemize}
\end{prop}

\section{Random substitution subshifts}\label{SEC:rs-subshift}
The theory of random substitutions is still in its infancy, and so far there is no canonical formulation. In what follows, we partially reuse notation and conventions from \cite{RS:random}. For an alternative formulation, see also the thesis of the first author \cite{G:gohlke-thesis}.

A function $\dsub \colon \mc A \to \mc A^+$ is called a \emph{deterministic substitution} and uniquely induces a monoid homomorphism $\dsub \colon \mc A^\ast \to \mc A^\ast$. Deterministic substitutions are extremely well-studied \cite{bible, F:book}.
In contrast, a random substitution can take multiple values on a single letter $a \in \mc A$. 
\begin{definition}
Let $\mc A $ be a finite alphabet, $\mathcal P (\mc A^{+})$ the power set of $A^{+}$ and define $\mc S (\mc A^{+}) = \{ B \in \mathcal P (A^{+}) \mid  B \text{ finite} \} \setminus \varnothing$.
A \emph{random substitution} on $\mc A$ is a map $\vartheta \colon \mc A \to \mathcal S(\mc A^{+})$.
We call a word $u$ a \emph{realisation} of $\vartheta$ on $a$ if $u \in \vartheta(a)$, in this case we also use the notation $u \overset{\bullet}{=} \vartheta(a)$. We may also write $\vartheta^k(v) \overset{\bullet}{=} \vartheta^m(w)$ if there is a common realisation (that is, their intersection is non-empty).
If $v \triangleleft u$, with $u \overset{\bullet}{=} \vartheta(a)$, then we write $v \blacktriangleleft \vartheta(a)$.
\end{definition}
A word which can appear as a realisation $u \overset{\bullet}{=}\vartheta(a)$ for some $a\in \mc A$ is sometimes called an \emph{inflation word}.
\begin{remark}
Note that in our definition we make no reference to the realisation probabilities of the elements in $\vartheta(a)$, $a \in \mc A$, unlike in \cite{RS:random}.
The subshift associated with $\vartheta$ is independent of the choice of such probabilities, as long as they are non-degenerate.
This justifies their omission here, where we study only their topological dynamics.
\end{remark}
We can extend $\vartheta$ to a function $\mc A^{+} \to \mc S(\mc A^+)$ by concatenation,
\[
\vartheta(a_1 \cdots a_m) = \vartheta(a_1) \cdots \vartheta(a_m) := \{ u_1 \cdots u_m \mid u_i \in \vartheta(a_i), \; \text{for } i \in \{ 1, \ldots, m \} \},
\]
and consequently to a function $\vartheta \colon \mathcal S (\mc A^{+}) \to \mc S (\mc A^+)$ by $\vartheta(B) := \bigcup_{u \in B} \vartheta(u)$.
This then allows us to take powers of $\vartheta$ by composition, giving $\vartheta^k \colon \mathcal S (\mc A^{+}) \to \mathcal S (\mc A^{+}) $, for any $k \geq 0$, where $\vartheta^0 := \operatorname{id}_{\mc S(\mc A^+)}$ is the identity and $\vartheta^{k+1} := \vartheta \circ \vartheta^k$.
We say that a word $u \in \mc A^{\ast}$ is $\vartheta$-legal, if there is a $k \in \N$ and an $a \in \mc A$ such that $u \blacktriangleleft \vartheta^k(a)$.
\begin{example}[Random Fibonacci]\label{EX:fib}
The \emph{random Fibonacci substitution} is defined on the alphabet $\mc A = \{a,b\}$ by
\[
\vartheta \colon a \mapsto \{ab,ba\}, \: b \mapsto \{a\}.
\]
The next two iterates of the random Fibonacci substitution are then given by

\[
\begin{array}{rl}
\vartheta^2 \colon & a \mapsto \{aba,baa,aab\}, \:\: b \mapsto \{ab,ba\}\\
\vartheta^3 \colon & a \mapsto \{abaab,ababa,baaab,baaba,aabab,aabba,abbaa,babaa\}, \:\: b \mapsto \{aba,baa,aab\}.
\end{array}
\]
So a realisation of $\vartheta^3(a)$ is given by $baaab \overset{\bullet}{=}\vartheta^3(a)$, hence the word $aaa \blacktriangleleft\vartheta^3(a)$ is $\vartheta$-legal.
\end{example}
\begin{definition}
The \emph{language} of a random substitution $\vartheta$ is the set of $\vartheta$-legal words, that is
\[
\mc L_{\vartheta} = \{u \blacktriangleleft \vartheta^k(a) \mid k \geq 0, a \in \mc A  \}.
\]
\end{definition}
From this, we construct a subshift associated with $\vartheta$, similar to the deterministic setting.
\begin{definition}
Let $\vartheta$ be a random substitution on a finite alphabet $\mc A$. Then the \emph{random substitution subshift} of $\vartheta$ (abbreviated \emph{RS-subshift}) is given by
\[
X_{\vartheta} = \{ w \in \mc A^{\Z} \mid u\triangleleft w \Rightarrow u \in \mc L_{\vartheta} \}.
\]
\end{definition}
It is easy to see that $X_{\vartheta}$ indeed forms a subshift---a closed, shift-invariant subspace of the full shift $\mc A^\Z$.
RS-subshifts are therefore compact.

Certain regularity properties of random substitutions make the study of their RS-subshifts more accessible.
The following definition is satisfied by many interesting examples of random substitutions, and certainly all those considered here.
\begin{definition}
A random substitution $\vartheta$ on a finite alphabet $\mc A$ is called \emph{primitive} if there exists a $k \in \N$ such that for all $a_i, a_j \in \mc A$ we have $a_i \blacktriangleleft \vartheta^k(a_j)$.
If $\vartheta$ is a primitive random substitution, then we call the associated subshift $X_{\vartheta}$ a \emph{primitive RS-subshift}.
\end{definition}
Note that, in contrast to the deterministic case, primitivity of a random substitution $\vartheta$ is not enough to conclude that the subshift $X_{\vartheta}$ is non-empty.

Many results relating to the dynamics and topology of primitive RS-subshifts were presented in previous work of two of the authors \cite{RS:random}.
We summarise some of these results in the next proposition\footnote{For a definition of splitting pairs, we refer the reader to \cite{RS:random}.}.
\begin{prop}\label{PROP:rs-subs}
Let $\vartheta$ be a random substitution with associated RS-subshift $X_\vartheta$. Then:
\begin{itemize}
\item $X_\vartheta$ is closed under substitution. That is, if $w \in X_\vartheta$, then $v \in X_\vartheta$ for all realisations $v \overset{\bullet}{=} \vartheta(w)$.
\item $X_\vartheta$ is closed under taking preimages. That is, if $w \in X_\vartheta$, then there exists a $k \geq 0$ and an element $w' \in X_\vartheta$ such that $S^k(w) \overset{\bullet}{=} \vartheta(w')$.
\item If $\vartheta$ is primitive, then:
\begin{itemize}
\item $X_\vartheta$ is empty if and only if, for all $a \in \mc A$ and all realisations $u \overset{\bullet}{=}\vartheta(a)$,  we have $|u| = 1$.
\item $X_\vartheta$ contains a dense orbit.
\item $X_\vartheta$ is topologically transitive.
\item $X_\vartheta$ is either finite or homeomorphic to a Cantor set.
\item The set of periodic elements in $X_\vartheta$ is either empty or dense.
\item $X_\vartheta$ contains a dense subset of linearly repetitive elements.
\item $X_\vartheta$ is either minimal or contains infinitely many distinct minimal subspaces.
\item If $\vartheta$ admits a splitting pair, then $h_{\operatorname{top}}(X_\vartheta) > 0$.
\end{itemize}
\end{itemize}
\end{prop}

We also provide here a new result related to the topological dynamics of RS-subshifts.
First, we need to introduce what it means for a subshift to be substitutive.

\begin{definition}
Let $X$ be a subshift.
An element is called \emph{linearly repetitive} if there exists a real number $L > 0$ such that for all $n \geq 1$ and for all $u \in \mc L^n(X)$, $v \in \mc L^{Ln}(X)$ one has $u \triangleleft v$.
We say that a subspace $Y\subset X$ is \emph{linearly recurrent} if $Y = \overline{\mc O(x)}$ for a linearly repetitive element $x\in X$.

A subspace $Y \subset X$ is called \emph{substitutive} if there exists a primitive deterministic substitution $\dsub$ such that $Y = X_\dsub$.
\end{definition}
All substitutive subshifts are linearly recurrent and all linearly recurrent subshifts are minimal.
The converse implications are false in general.

\begin{prop}\label{PROP:min-subspaces}
Let $\vartheta$ be a primitive random substitution with non-empty RS-subshift $X_\vartheta$.
Either $X_\vartheta$ is substitutive or there are infinitely many distinct substitutive subspaces of $X_\vartheta$.
\end{prop}
\begin{proof}
Recall that, as $\vartheta$ is primitive, $X_\vartheta$ contains a point $w$ with dense orbit.
Suppose $X_\vartheta$ is not substitutive.
It is clear, by restricting a suitably large power of the random substitution to a primitive deterministic sub-substitution, that $X_\vartheta$ contains at least one substitutive subspace.
As substitutive subshifts are minimal, and $X_\vartheta$ properly contains a substitutive subspace, it follows that $X_\vartheta$ is not minimal.
Suppose $X_\vartheta$ has at least $n$ minimal subspaces $X_1, \ldots, X_n$.
As $w$ has a dense orbit and $X_\vartheta$ is not minimal, we know that $w \notin X_i$ for any $1 \leq i \leq n$.
As all of the sets $X_i$ are closed, and there are only finitely many, $A = \bigcup_{i=1}^n X_i$ is also a closed set.
It follows that $U = X_\vartheta \setminus A$ is a non-empty open set.
It is now enough to show that the union of the substitutive subspaces of $X_\vartheta$ forms a dense subset.

Let $u \in \mc L_\vartheta$ be a word admitted by $\vartheta$ such that the cylinder set $\mc Z_0(u) = \{x \in X_\vartheta \mid x_{0}\cdots x_{|u|-1} = u\}$ is contained in $U$ and so that $u$ contains every letter in $\mc A$---this can always be done as $u$ can be chosen to be a suitably long subword of $w$ and $w$ contains every letter.
We will show that a substitutive subspace of $X_\vartheta$ exists which intersects $\mc Z_0(u)$ and so cannot be any of the subspaces $X_1, \ldots, X_n$.
As $u$ is admitted by $\vartheta$, there exists a natural number $k\geq 1$ and a letter $a \in \mc A$ such that $u \blacktriangleleft \vartheta^k(a)$.
Let $\hat{u}$ be the particular realisation of $\vartheta^k(a)$ which contains $u$ as a subword.
Without loss of generality, suppose that $k$ is large enough so that every letter $a' \in \mc A$ has a realisation of $\vartheta^k(a')$ which contains every letter in $\mc A$ (if not, using primitivity, we can increase $k$ to a suitably large natural number so that this is the case).
We then define a substitution $\dsub$ on $\mc A$ which maps $a$ to $\hat{u}$ and maps all other letters $a'$ to some realisation which contains every letter in $\mc A$.
The substitution $\dsub$ is then primitive by definition and, by construction, $X_\dsub \subset X_\vartheta$.
It is also the case that $u \in \mc L_\dsub$ as $u \triangleleft \hat{u} = \dsub(a)$.
It follows from primitivity of $\dsub$ that $X_\dsub \cap \mc Z_0(u) \neq \varnothing$.
Hence, $X_\vartheta$ contains $n+1$ substitutive subspaces and so, by induction, contains infinitely many.
\end{proof}

\begin{example}
To illustrate the construction used in the proof of Proposition \ref{PROP:min-subspaces}, let us consider the random Fibonacci substitution of Example \ref{EX:fib}
\[
\vartheta \colon a \mapsto \{ab,ba\}, \: b \mapsto \{a\}.
\]
Recall that the word $aaa$ is $\vartheta$-legal, as $aaa \blacktriangleleft \vartheta^3(a)$. In particular, if we choose the realisation $baaab \overset{\bullet}{=}\vartheta^3(a)$ and arbitrarily choose a realisation for $\vartheta^3(b)$, for instance $aba \overset{\bullet}{=}\vartheta^3(b)$, then the primitive deterministic substitution $\dsub \colon a \mapsto baaab, \: b \mapsto aba$ has an associated subshift $X_\dsub$ which is a substitutive subspace of $X_\vartheta$ and which contains the $\vartheta$-legal word $aaa$ in its language.

By increasing the power $k$ of $\vartheta^k$ and considering all possible globally constant realisations, we produce an infinite family of deterministic substitutions whose subshifts are necessarily substitutive subspaces of $X_\vartheta$ and whose union forms a dense subset of $X_\vartheta$. Note that these substitutive subspaces are neither necessarily distinct, nor exhaustive. For instance, both the usual Fibonacci substitution $\dsub \colon a \mapsto ab, \: b \mapsto a$ and its reverse $\overline{\dsub} \colon a \mapsto ba, \: b \mapsto a$ have the same associated subshifts $X_\dsub = X_{\overline{\dsub}}$, and are both globally constant realisations of the random Fibonacci substitution (with $k=1$) \cite{GM:random-fib}.

For an example of a random substitution where this process certainly does not produce all substitutive subspaces, consider the substitution
\[
\vartheta \colon a \mapsto \{aa,ab,ba,bb\}, \: b \mapsto \{aa,ab,ba,bb\}
\]
whose associated RS-subshift is the full shift $X_\vartheta = \{a,b\}^\Z$, as shown in \cite{RS:random}. It follows that every substitution subshift on two letters is a subspace of $X_\vartheta$, however the only substitutions which can be built by taking a globally constant realisations of $\vartheta^k$ for some $k$ are necessarily of \emph{constant length} (there exists an integer $\ell\geq 2$ such that $|\dsub(a)|=\ell$ for all $a \in \mc A$). In particular, no substitutive subshift whose relative word frequencies are irrationally related can appear this way (such as the Fibonacci substitution) \cite[Sec 5.4]{Q:book}.
\end{example}

\section{Main result}\label{SEC:main}
Let $A$ be an irreducible $0$-$1$ matrix so that $X_{A}$ is a topologically transitive SFT.
As remarked in Section \ref{SEC:sft}, the set of periodic points $\operatorname{Per}(X_{\mc A})$ forms a dense subset of $X_{A}$.
In order to realise $X_{A}$ as an RS-subshift, our strategy will be to construct a random substitution $\vartheta$ such that $\operatorname{Per}(X_{A}) \subseteq X_{\vartheta}$ and in such a way that $\mc L_{\vartheta}$ contains no forbidden words in $\mc F_A$.
That is, we should also have the inclusion $X_{\vartheta} \subseteq X_{A}$.
By density of $\operatorname{Per}(X_{A})$ in $X_{A}$ and the compactness of $X_\vartheta$, it follows that $X_{A} = \overline{\operatorname{Per}(X_{A})} \subseteq \overline{X_{\vartheta}} = X_{\vartheta}$ and hence that $X_\vartheta = X_{A}$.

\begin{definition} \label{DEF:G-rs}
Let $\mc A$ be a finite alphabet and let $G = (V,E)$ be a strongly connected graph with $V = \{ v_a \mid a \in \mc A\}$. Given a cycle $\gamma = ((v_{a_1}, v_{a_2}), (v_{a_2}, v_{a_3}), \ldots, (v_{a_\ell}, v_{a_1}))$ in $G$, the \emph{word read} $u(\gamma) \in \mc A^+$ of $\gamma$ is the word $a_1 a_2 \cdots a_\ell a_1$, given by reading the vertices of $G$ traversed by the cycle $\gamma$ from the root back to itself.
For every letter $a \in \mc A$, let $C_a = \{\gamma_1^a, \ldots, \gamma_{k_a}^a\}$ be the set of simple cycles in $G$ with root $v_a$.
We let $\vartheta_G \colon \mc A \to \mc S(\mc A^+)$ denote the \emph{cycle-substitution} associated with $G$, given by
\[
\vartheta_G \colon a \mapsto \{a\} \cup \{u(\gamma) \mid \gamma \in C_a\}.
\]
\end{definition}

\begin{remark}
As every simple cycle has length at most $n = |\mc A|$, $C_a$ is finite.
So $\vartheta_G$ is a well-defined random substitution.
\end{remark}

\begin{lemma}\label{LEM:subcycle_by_subst}
Let $G$ be a strongly connected graph, and let $\gamma$ be a cycle in $G$ that is not simple.
Suppose $\gamma_0$ is a simple subcycle of $\gamma$, and $\tilde{\gamma}$ is the cycle given by removing $\gamma_0$ from $\gamma$.
Then, $u (\gamma)\overset{\bullet}{=} \vartheta_G(u(\tilde{\gamma}))$.
\end{lemma}
\begin{proof}
Suppose $\gamma = ((v_{a_0}, v_{a_1}), (v_{a_1}, v_{a_2}), \ldots, (v_{a_\ell}, v_{a_0}))$ and $\gamma_0 = ((v_{a_i},v_{a_{i+1}}), \ldots, (v_{a_{j-1}},v_{a_j}))$, with $a_i = a_j$.
Then, $\tilde{\gamma} = ((v_{a_0}, v_{a_1}), \ldots, (v_{a_{i-1}}, v_{a_i}),  (v_{a_i}, v_{a_{j+1}}), \ldots (v_{a_{\ell}}, v_{a_0}))$.

The corresponding word reads are given by $u(\gamma) = a_0 a_1 \cdots a_{\ell} a_0$, $u(\gamma_0) = a_i \cdots a_j$, and finally $u (\tilde{\gamma}) = a_0 \cdots a_i a_{j+1} \cdots a_{\ell} a_0$.
Since $\gamma_0$ is simple with root in $a_i$, it is contained in $C_{a_i}$, thus $u (\gamma_0) \overset{\bullet}{=} \vartheta_G(a_i)$.
We choose a specific realisation $\vartheta_G^R$ of $\vartheta_G$ on $u(\tilde{\gamma})$ by: 
\[
\vartheta_G^R(u(\tilde{\gamma})_k) = 
\begin{cases}
u(\tilde{\gamma})_k, & k \neq i,
\\u(\gamma_0), & k = i.
\end{cases}
\]
With this choice, $\vartheta_G^R(u(\tilde{\gamma})) = a_0 \cdots a_{i-1} a_i \cdots a_j a_{j+1} \cdots a_{\ell} a_0 = u(\gamma)$, hence $u (\gamma)\overset{\bullet}{=} \vartheta_G(u(\tilde{\gamma}))$. 
\end{proof}

\begin{theorem}\label{THM:main}
Let $A$ be an irreducible $0$-$1$ matrix with associated SFT $X_{A}$ over the alphabet $\mc A$.
There exists a primitive random substitution $\vartheta \colon \mc A \to \mc S(\mc A^+)$ with associated RS-subshift $X_\vartheta$ such that $X_{\vartheta} = X_{A}$.
\end{theorem}
\begin{proof}

Let $\vartheta:= \vartheta_{G_A}$, with $G_A$ the strongly connected graph associated to the matrix $A$. We first show that $X_{\vartheta} = X_{A}$. 
We note that for all $a \in \mc A$, and all realisations of $\vartheta(a)$, no forbidden word in $\mc F_A$ appears as a subword of $\vartheta(a)$ (because every word $\vartheta(a)$ corresponds to a path in $G_{A}$).
As it is also the case that every realisation of $\vartheta(a)$ is either of the form $a$ or $aua \overset{\bullet}{=} \vartheta(a)$ for some $u$, it follows that no forbidden word in $\mc F_A$ can appear as a subword of a realisation of an iterated substituted letter $\vartheta^k(a)$ for all $k$.
This is because forbidden words cannot be created at the boundaries of concatenated substituted letters, as boundaries between letters remain fixed under substitution for all realisations.
We hence have the inclusion $X_\vartheta \subseteq X_{A}$.

Let $w \in X_{A}$ be a periodic point.
So there exists an $\ell \geq 1$ such that $\sigma^\ell(w) = w$.
Let $w_0 \cdots w_{\ell-1}$ be a periodic block of $w$.
The legal word $v = w_0 \cdots w_{\ell - 1} w_0$ corresponds to a cycle $\gamma_v$ in the graph $G_{A}$.
More precisely, $v = u(\gamma_v)$ for some $\gamma_v$ in $G_{A}$.
Due to Proposition \ref{PROP:cycle-decomp}, there are cycles $\gamma_0, \ldots, \gamma_k$ and $\tilde{\gamma_0}, \ldots \tilde{\gamma}_{k+1}$ such that $\gamma_v = \tilde{\gamma}_0$, all the $\gamma_i$ and $\tilde{\gamma}_{k+1}$ are simple and $\tilde{\gamma}_{j+1}$ is given by removing $\gamma_j$ from $\tilde{\gamma}_{j}$ for all $j \in \{1,\ldots,k\}$.
By Lemma \ref{LEM:subcycle_by_subst}, that implies $u(\tilde{\gamma}_{j}) \overset{\bullet}{=} \vartheta (u(\tilde{\gamma}_{j+1}))$ for all $j \in \{1,\ldots,k\}$.
A simple iteration of this relation yields $u(\gamma_v) = u(\tilde{\gamma}_0) \overset{\bullet}{=} \vartheta^{k+1}(\tilde{\gamma}_{k+1})$.
Since $\tilde{\gamma}_{k+1}$ is a simple cycle, we have $u(\tilde{\gamma}_{k+1}) \overset{\bullet}{=} \vartheta(a)$ for some $a \in \mc A$ by the definition of $\vartheta$.
It follows that $u(\gamma_v) \overset{\bullet}{=} \vartheta^{k+2}(a)$.

This shows that $v = u(\gamma_v)$ is an admitted word for $\vartheta$ and so $\mc L(\operatorname{Per}(X_{A})) \subset \mc L_\vartheta$. By the density of $\operatorname{Per}(X_{A})$ in $X_{A}$, every legal word in $\mc L(X_{A})$ can be extended to a legal periodic block, so it follows that $\mc L(X_{A}) \subseteq \mc L_\vartheta$ and so $X_{A} \subseteq X_\vartheta$.

It remains to be shown that $\vartheta_{G_A}$ is primitive.
Let $a,b$ be letters in $\mc A$ and let $\gamma_{ab}$ be a cycle rooted at $a$ which passes through $b$.
Such a cycle exists by Lemma \ref{LEM:cycles}.
By the above, there exists an integer $k_{ab} \geq 0$ and a letter $c \in \mc A$ such that $u(\gamma_{ab}) \overset{\bullet}{=} \vartheta^{k_{ab}}(c)$.
Note that $b \triangleleft u(\gamma_{ab})$ by construction of the cycle $\gamma_{ab}$ and the definition of the word read of a path.
Of course, $c$ must actually be the letter $a$, as $\gamma_{ab}$ is rooted at $a$ so the end letters of its word read $u(\gamma_{ab})$ are both $a$, and the substitution $\vartheta$ fixes the end letters of words.
Hence, $b \blacktriangleleft \vartheta^{k_{ab}}(a)$.
As we always have $a \blacktriangleleft \vartheta^i(a)$ for all $i \geq 1$ and $a \in \mc A$, it follows that we also have $b \blacktriangleleft \vartheta^{k_{ab} + i}(a)$ for all $i \geq 1$.
Let $k = \max \{k_{ab} \mid a,b \in \mc A\}$.
By construction then, for all $a,b$ we have $b \blacktriangleleft \vartheta^k(a)$ and so $\vartheta$ is primitive. 
\end{proof}
\begin{coro}\label{COR:main}
Let $X$ be a topologically transitive SFT over the alphabet $\mc A$.
There exists a primitive random substitution $\vartheta$ (on a possibly different alphabet) with associated RS-subshift $X_\vartheta$ such that $X_{\vartheta}$ is topologically conjugate to $X_{A}$.
\end{coro}
\begin{proof}
It is well known that every topologically transitive SFT can be rewritten up to topological conjugacy as an SFT associated to an irreducible $0$-$1$ matrix \cite[Prop.~2.3.9]{LM:introduction-to-symbolic}, and so the result follows from a direct application of Theorem \ref{THM:main}.
\end{proof}

\begin{remark}
We note that Theorem \ref{THM:main} and Corollary  \ref{COR:main} are constructive:
Given a topologically transitive SFT $X_{\mc F}$, construct a topologically conjugate SFT $X_A$ (in the standard way) whose associated transition matrix $A$ is irreducible with entries in $\{0,1\}$.
One then forms the directed graph $G_A$ and can explicitly read off the associated random substitution $\vartheta_G$.
Note, however, that $\vartheta_G$ is far from being unique in this respect.
In any given example, there will almost certainly be a `simpler' random substitution that realises $X_{\mc F}$ as the corresponding RS-subshift.
\end{remark}

In order to illustrate the general construction, we provide an example.
\begin{example}\label{EX:mickey}
Let $\mc A = \{0,1,2,3\}$ and consider the following irreducible $0$-$1$ matrix.
\[
A =
\begin{bmatrix}
1 & 1 & 0 & 0 \\
0 & 0 & 1 & 1 \\
1 & 1 & 0 & 0 \\
0 & 1 & 0 & 0
\end{bmatrix}
\]
which is irreducible and thus corresponds to a topologically transitive SFT $X_A$.
The corresponding graph $G := G_{A}$ associated with $A$ is shown in Figure \ref{FIG:graph-mickey}.
We drop at this point the formal distinction between a letter and the corresponding vertex to ease readability.
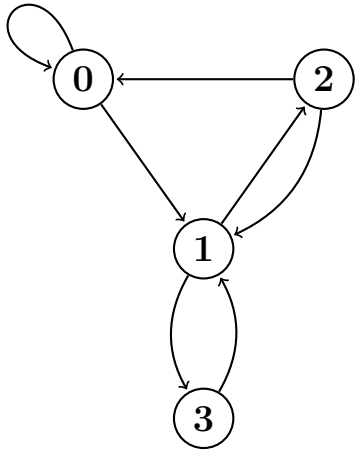
\begin{figure}[h]
\begin{tikzpicture}[->, shorten >=1pt, thick, main node/.style={circle,draw,font=\Large\bfseries}, scale = 0.8]

\node[main node](1) at (0,0) {0};
\node[main node](2) at (2,-2.82) {1};
\node[main node](3) at (4,0) {2};
\node[main node](4) at (2,-5.64) {3};

\draw (1) to[out = 110, in = 160, looseness=10] (1);
\draw (1) to (2);
\draw (2) to (3);
\draw (3) to[bend left] (2);
\draw (3) to (1);
\draw (2) to[bend right] (4);
\draw (4) to[bend right] (2);

\end{tikzpicture}
\caption{Graph $G_{A}$ of the SFT $X_{A}$ in Example \ref{EX:mickey}.}
\label{FIG:graph-mickey}
\end{figure}

The cycle-substitution associated with $G$ is given by
\[
\vartheta_{G} \colon
\begin{cases}
0 \mapsto & \{0, 00, 0120\},\\
1 \mapsto & \{1, 1201, 121, 131\},\\
2 \mapsto & \{2, 212, 2012\},\\
3 \mapsto & \{3, 313\}.
\end{cases}
\]
Note that, by construction, the set of paths corresponding to the words $\vartheta(a) \neq a$ for all $a \in \mc A$ are precisely the simple cycles in $G$ with root at the vertex $a$. 

With the same arguments as provided in the proof of Theorem \ref{THM:main} (no forbidden words can appear within or at the boundaries of inflation words), we easily observe the inclusion $X_\vartheta \subseteq X_{A}$.

We illustrate by an example how to iteratively construct a periodic word in $X_{A}$ via $\vartheta_{G}$. Consider the word $v = 213120012$, corresponding to a cycle $\gamma$ in $G_{A}$.
The bi-infinite sequence $\omega = (v')^\infty = (21312001)^\infty$, is therefore a periodic point in $X_{A}$.
Our aim is to show that $v'$ is a legal word for the random substitution $\vartheta$.
First, we identify a simple subcycle of $\gamma$, for example the one whose word read is $131$.
Since $131 \overset{\bullet}{=} \vartheta(1)$ and $i \overset{\bullet}{=} \vartheta(i)$ for $i \in \{0,1,2,3 \}$, we have $v = 213120012 \overset{\bullet}{=} \vartheta(2\dot{1}20012)$ (where we identify with a dot $\dot{\quad}$ a letter which is to be substituted non-trivially).
Iterating this procedure, we find
\[
v = 213120012 \overset{\bullet}{=} \vartheta(2\dot{1}20012) \overset{\bullet}{=} \vartheta^2(\dot{2}0012) \overset{\bullet}{=} \vartheta^3(2\dot{0}12) \overset{\bullet}{=} \vartheta^4(\dot{2})
\]
This yields $v \overset{\bullet}{=} \vartheta^{4}(2)$, showing that $v'$ is $\vartheta$-legal.
As a similar process can be followed to show that $(v')^k$ is legal for any $k$, it follows that $\omega = (v')^\infty$ is an element of $X_\vartheta$.
\end{example}

\begin{definition}
Let $\alpha \geq 1$ be a real algebraic number.
If all other algebraic conjugates of $\alpha$ are strictly smaller in absolute value than $|\alpha|$, then we call $\alpha$ a \emph{Perron number}.
\end{definition}
Lind \cite{L:perron-numbers} showed that $\alpha$ is a Perron number if and only if there exists a primitive integer matrix whose Perron--Frobenius eigenvalue is $\alpha$.
As every primitive matrix $A$ corresponds to a topologically mixing (hence transitive) shift of finite type $X_{\mc F}$ whose topological entropy $h_{\operatorname{top}}$ is given by $h_{\operatorname{top}}(X_{\mc F}) = \log \alpha$ where $\alpha$ is the Perron--Frobenius eigenvalue of $A$, we have the following simple corollary of Theorem \ref{THM:main}.

\begin{coro} \label{COR:Perron-numbers}
For every Perron number $\alpha$, there exists a primitive random substitution $\vartheta$ whose associated RS-subshift $X_\vartheta$ is topologically mixing and whose topological entropy $h_{\operatorname{top}}(X_\vartheta)$ is given by
\[
h_{\operatorname{top}}(X_\vartheta) = \log (\alpha).
\]
\end{coro}

Theorem \ref{THM:main} also suggests a rich structure of minimal subspaces for $X_A$.
In particular, every primitive substitution given by taking a particular realisation of $\vartheta_{G_A}^k$ for some $k$ corresponds to a minimal subspace of $X_{\vartheta_{G_A}} = X_A$.
The union of such minimal sets provides a dense subset of $X_A$ \cite{RS:random} and so we can `internally approximate' SFTs by successive substitutive subspaces.

\subsection{Edge Shifts}
Often, edge shifts are a more convenient way of representing SFTs, rather than vertex shifts.
For a detailed introduction to edge shifts, we suggest the text by Lind and Marcus \cite{LM:introduction-to-symbolic}.
An edge shift is similar to a vertex shift, except one uniquely labels the edges of a directed graph $G$ (where we now allow multiple edges between vertices) and then the corresponding alphabet is given by the set of edge-labels $E$.
A bi-infinite path in $G = (V,E)$ then corresponds to the bi-infinite sequence in the edge shift $X_G$ given by reading the sequence of edges traversed by the path (we use a hat $\widehat{\quad}$ to distinguish between vertex shifts $\widehat{X}_G$ and edge shifts $X_G$),
\[
X_G = \{ x \in E^\Z \mid t(x_i) = s(x_{i+1})\}.
\]
As with vertex shifts, $G$ is strongly connected if and only if the edge shift $X_G$ is topologically transitive and every SFT is topologically conjugate to an edge shift on some finite directed graph $G$.

Up to topological conjugacy, it makes little difference if vertex shifts or edge shifts are used when rewriting a transitive SFT as a primitive RS-subshift, although it may be more convenient in some circumstances to use the edge shift representation.
The only difference is that one needs to use a slightly modified definition of a simple cycle, whereby an \emph{edge-wise simple cycle} is given by a cycle that traverses each \emph{edge} at most once, rather than each vertex.
Likewise, a modified version of Lemma \ref{LEM:simple-cycles} and Proposition \ref{PROP:cycle-decomp} are needed, allowing one to decompose cycles into nested edge-wise simple subcycles.
The word read $u(\gamma)$ of a cycle $\gamma$ is then as usual, except one reads the edges traversed by $\gamma$ rather than vertices, and one includes the first edge at both the beginning and end of the word read.
So, if $\gamma = (e_1, \ldots, e_\ell)$ is a cycle, then $u(\gamma) = e_1\cdots e_\ell e_1$.
The \emph{edge-wise cycle-substitution} is then again given by mapping
\[
\vartheta_G \colon a \mapsto \{a\} \cup \{u(\gamma) \mid \gamma \in C_a\}
\]
where $C_a$ is now the set of edge-wise simple cycles in $G$ whose first edge is $a$.

Rather than rigorously outlining this perspective, we instead provide an example computation and invite the interested reader to complete the necessary formalities.
\begin{example}\label{EX:edge-shift}
Let $G$ be given by the labelled directed graph in Figure \ref{FIG:graph-mickey2}.
\begin{figure}[h]
\begin{tikzpicture}[shorten >=0.5pt]

\node[fill, circle](1) at (0,0) {};
\node[fill, circle](2) at (2,3.4) {};
\node[fill, circle](3) at (4,0) {};

\draw[->, thick, bend left, left] (1) to node {$0$} (2);
\draw[->, thick, bend right, left] (1) to node {$1$} (2);
\draw[->, thick, bend left, right] (2) to node {$2$} (3);
\draw[->, thick, bend left, below] (3) to node {$3$} (1);
\draw[->, thick, out = 10, in = 170, above] (1) to node {$4$} (3);
\draw[->, thick, out = -40, in = 40, looseness=20, right] (3) to node {$5$} (3);

\end{tikzpicture}
\caption{Graph $G$ with labelled edges for Example \ref{EX:edge-shift}.}
\label{FIG:graph-mickey2}
\end{figure}
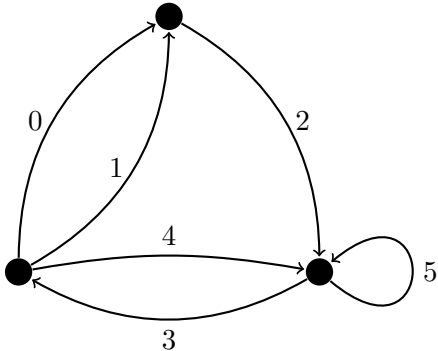
The set of edge-wise simple cycles beginning with the edge $0$ is given by $C_{0} = \{(0,2,3), (0,2,5,3)\}$ because any other tuple of edges either does not describe an admitted cycle in $G$ beginning with the edge $0$, or else contains an appearance of some edge more than once.
Likewise, we have
\[
\begin{array}{lcl}
C_0 & = & \{(0,2,3), (0,2,5,3)\}, \\
C_1 & = & \{(1,2,3), (1,2,5,3)\}, \\
C_2 & = & \{(2,3,0), (2,3,1), (2,5,3,0), (2,5,3,1)\}, \\
C_3 & = & \{(3,0,2), (3,0,2,5), (3,1,2), (3,1,2,5), (3,4), (3,4,5)\}, \\
C_4 & = & \{(4,3), (4,5,3)\}, \\
C_5 & = & \{(5), (5,3,0,2), (5,3,1,2), (5,3,4)\}.
\end{array}
\]
The corresponding cycle-substitution associated with $G$ is given by
\[
\vartheta:= \vartheta_{G} \colon
\begin{cases}
0 \mapsto & \{0, 0230, 02530\},\\
1 \mapsto & \{1, 1231, 12531\},\\
2 \mapsto & \{2, 2302, 2312, 25302, 25312\},\\
3 \mapsto & \{3, 3023, 30253, 3123, 31253, 343, 3453\},\\
4 \mapsto & \{4, 434, 4534\},\\
5 \mapsto & \{5, 55, 53025, 53125, 5345\} .
\end{cases}
\]
We claim that $X_{\vartheta} = X_G$.
To illustrate how to use the cycle decomposition, consider the example cycle $\gamma = (0,2,3,1,2,5,3,4,5,3)$ whose corresponding word read $u := u(\gamma)$ is given by $u = 02312534530$.
As $\gamma$ is a cycle, the bi-infinite sequence $(0231253453)^\infty$ is a periodic element of $X_G$ and $u$ is a word in the language $\mc L(X_G)$.
We can identify the simple subcycle $(3,1,2,5)$ in $\gamma$ and so $u \overset{\bullet}{=} \vartheta(02\dot{3}4530)$ by realising $\vartheta$ as the identity on each letter except the left-most $3$ which should be substituted as $31253 \overset{\bullet}{=} \vartheta(3)$ (where we identify with a dot $\dot{\quad}$ a letter which is to be substituted non-trivially).
Continuing, we see that
\[u = 02312534530 \overset{\bullet}{=} \vartheta(02\dot{3}4530) \overset{\bullet}{=} \vartheta^2(02\dot{3}0) \overset{\bullet}{=} \vartheta^3(\dot{0})\]
and so $u \blacktriangleleft \vartheta^3(0)$, hence $u \in \mc L_\vartheta$.

As with the vertex shift proof, a similar method works for all periodic words in $X_G$ and so $X_G \subseteq X_{\vartheta}$.
The opposite inclusion is obvious, given the definition of $\vartheta_G$ in terms of simple cycles. Hence, $X_{\vartheta} = X_G$.
\end{example}
Edge shifts are the natural setting for studying \emph{sofic shifts} which are exactly those subshifts $Y$ that admit a factor map $f \colon X \to Y$ where $X$ is an SFT (in particular, every SFT is trivially sofic).
Equivalently, a sofic shift is an edge shift on a finite graph where there is no uniqueness condition restricting the possible edge-labels---that is, multiple edges can share the same label.
Although we can handle edge shifts in a similar manner to vertex shifts when $G$ has all of its edges uniquely labelled, there is an obstruction to this extending to the case of general sofic edge shifts.

In particular, the process of iterating a cycle-substitution on letters associated with a graph no longer mimics the cycle decomposition result of Proposition \ref{PROP:cycle-decomp}.
Essentially, the substitution is blind to the particular edges of a graph that a letter might be associated with, unless the identically-labelled edges are indistinguishable (for instance if they are labelled uniquely as in the case of SFTs). Therefore,  iterating the substitution will create illegal words by trying to insert a word read of a simple cycle $\gamma_1$ which cannot be nested in another cycle $\gamma_2$ although the \emph{label} of the first edge of $\gamma_1$ also appears at an edge of $\gamma_2$.
This is illustrated in the following example.
\begin{example}\label{EX:sofic}
Let $G$ be the labelled directed graph in Figure \ref{FIG:graph-mickey3} and let $X_G$ be the edge shift on $G$, hence a sofic shift.
We have used a hat $\widehat{\quad}$ to distinguish the two edges labelled by the letter $0$, but for the purposes of the subshift $X_G$, they are identical symbols, $0 = \widehat{0}$.
\begin{figure}[h]
\begin{tikzpicture}[shorten >=0.5pt]

\node[fill, circle](1) at (0,0) {};
\node[fill, circle](2) at (3,0) {};
\node[fill, circle](3) at (6,0) {};

\draw[->, thick, bend left, above] (1) to node {$0$} (2);
\draw[->, thick, bend left, above] (2) to node {$1$} (3);
\draw[->, thick, bend left, below] (3) to node {$2$} (2);
\draw[->, thick, bend left, below] (2) to node {$3$} (1);

\draw[->, thick, out = -40, in = 40, looseness=20, right] (3) to node {$\widehat{0}$} (3);

\end{tikzpicture}
\caption{Graph $G$ with labelled edges for Example \ref{EX:sofic}.}
\label{FIG:graph-mickey3}
\end{figure}
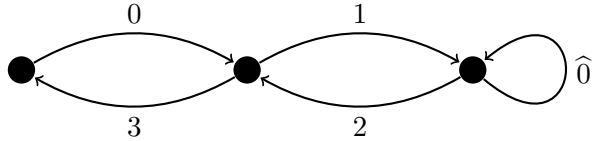
The edge-wise simple cycles beginning with with the edge $\widehat{0}$ are $C_{\widehat{0}} = \{(\widehat{0}), (\widehat{0},2,1), (\widehat{0},2,3,0,1)\}$.
One also identifies the simple cycle $(3,0) \in C_3$.
A reasonable interpretation of a cycle-substitution $\vartheta:= \vartheta_G$ on $\mc A = \{0,1,2,3\}$ associated with the graph $G$ would then have $a \overset{\bullet}{=} \vartheta(a)$ for all $a \in \mc A$, as well as $0210 \overset{\bullet}{=} \vartheta(0)$ and $303 \overset{\bullet}{=} \vartheta(3)$.
However, this would then give $302103 \overset{\bullet}{=} \vartheta(3\dot{0}3) \overset{\bullet}{=} \vartheta^2(\dot{3})$.
The word $302103$ is certainly not admitted by the language $\mc L(X_G)$ and so such an interpretation of a cycle-substitution cannot be correct for strictly sofic shifts.
At best, we can guarantee $X_G \subseteq X_{\vartheta_G}$.

This mistake is introduced by na\"{i}vely inserting the word read of the cycle $\gamma_1 = (\widehat{0},2,1)$, given by $u(\gamma_1) = 0210$, into the word read of the cycle $\gamma_2 = (3,0)$.
However, these two cycles are incompatible.
There is no cycle $\gamma$ in $G$ with simple subcycle $\gamma_1$ such that deleting $\gamma_1$ leaves $\gamma_2$.

\end{example}
It follows that, if sofic shifts can also be encoded as RS-subshifts, techniques different to those presented in the proof of Theorem \ref{THM:main} are needed (certainly, there exist purely sofic shifts which are also primitive RS-subshifts \cite{RS:random}).
\begin{question}
Is every topologically transitive sofic shift topologically conjugate to a primitive RS-subshift?
\end{question}

Principally, we have focussed in this article on the topological dynamics of RS-subshifts.
However, note that there are also important measure-theoretic aspects that one should study.
For instance, one can define shift-invariant measures on RS-subshifts $X_{\vartheta}$, which encode the relative word-frequencies of legal words for almost all elements in $X_\vartheta$ \cite{RS:random}.
Many of these measures turn out to be ergodic \cite{GS:ergodic}.
Since every topologically transitive SFT can be realised as an RS-subshift, the following question is then a natural one.

\begin{question}
Given a Markov measure on an SFT, does there exist a primitive random substitution (together with probability assignments) such that the frequency measure of the associated RS-subshift coincides with the Markov measure on the SFT?
For cycle-substitutions, is every such frequency measure also a Markov measure?
\end{question}

As we will soon shift to discussing topological entropy, it is also worth mentioning that metric entropy can play an important role in the ergodic theory of RS-subshifts.
Other than in the Ph.D thesis of Wing \cite{W:wing-thesis}, we are aware of no work in the literature which attempts to explore metric entropy of RS-subshifts.
Therefore, although we do not discuss the metric entropy here, we would be remiss for not stating the following natural questions which currently remain open.

\begin{question}
Are RS-subshifts intrinsically ergodic? That is, does every primitive RS-subshift admit a measure of maximal entropy and are they unique? Can a measure of maximal entropy be realised as a frequency measure?
\end{question}

Clearly, by Theorem \ref{THM:main}, we know of at least one family of intrinsically ergodic primitive RS-subshifts---namely those coming from cycle-substitutions---as the Parry measure is the unique measure of maximal entropy for SFTs \cite{P:intrinsic-markov}.
Whether Parry measures are ever frequency measures remains unclear.
For instance, some preliminary hand calculations indicate that Markov measures for the golden mean shift can be realised as frequency measures for an associated cycle-substitution, however this is far from a proof and remains speculative.

\section{Entropy}\label{SEC:entropy}
The goal of this section is to explicitly show that the possible values of topological entropy that can be realised for RS-subshifts are dense in the positive reals.
This is obviously a corollary of the main result in Section \ref{SEC:main}---compare Corollary~\ref{COR:Perron-numbers}.
However, the advantage here is that we rely only on elementary methods, rather than needing the full power of results related to Perron numbers and shifts of finite type.
One other advantage is that the results proved along the way are of independent interest and are not themselves corollaries of the main theorem of Section \ref{SEC:main}.

Throughout, for a random substitution $\vartheta$, let $h_\vartheta:= h_{\operatorname{top}}(X_\vartheta)$ denote the topological entropy of the associated RS-subshift $X_\vartheta$ and let $p_\vartheta(n):= |\mc L^n(X_\vartheta)|$ denote the corresponding complexity function.
Let
\[
\boldsymbol H = \{h_\vartheta \mid \vartheta \mbox{ is a primitive random substitution}\}
\]
denote the set of possible values of topological entropy for primitive RS-subshifts.
This is the set with which we are principally concerned.

The following proposition illustrates a general method (or family of methods) that can be used to construct for any primitive random substitution, a new primitive random substitution which factors onto the original, and where we can control the entropy of the new substitution in terms of the entropy of the original.
The basic idea is that we introduce new copies of letters from the original alphabet which are coupled with the original in controllable ways.
In the case of the next result, the new substitution can be constructed to have arbitrarily large entropy.
This kind of method will be used often in this section.
\begin{prop}\label{PROP:log2}
Let $\vartheta$ be a primitive random substitution on an alphabet $\mc A$ with RS-subshift $X_\vartheta$.
For all integers $m \geq 2$, there exists a primitive random substitution $\hat{\vartheta}$ and a factor map $f \colon X_{\hat{\vartheta}} \to X_\vartheta$ such that
\[
h_{\hat{\vartheta}} = h_\vartheta + \log (m).
\]
In particular, the set $\boldsymbol H$ is closed under addition by $\log (m)$.
\end{prop}
\begin{proof}
Let $\hat{\mc A} = \mc A \times \{1, \ldots, m\}$.
For every word $u$ in $\hat{\mc A}^n$, let $f \colon \hat{\mc A}^n \to \mc A^n$ be the projection map given by the one-block code $(a,i) \mapsto a$.
Define the substitution $\hat{\vartheta}$ by $\hat{\vartheta} \colon (a,i) \mapsto f^{-1}(\vartheta(a))$.
The substitution $\hat{\vartheta}$ is primitive because $\vartheta$ is primitive and the one-block code $f$ restricts to a factor map $X_{\hat{\vartheta}} \to X_\vartheta$.

If $u \in \mc L^n(X_\vartheta)$, then $|f^{-1}(u)| = m^n$ as the preimage of each letter in $u$ has $m$ possibilities and all such words are $\hat{\vartheta}$-legal by construction.
It follows that $p_{\hat{\vartheta}}(n) = m^n p_{\vartheta}(n)$.
We then calculate
\[
\begin{array}{rcl}
h_{\hat{\vartheta}} = \displaystyle\lim_{n \to \infty} \frac{\log \left(p_{\hat{\vartheta}}(n)\right)}{n}&=& \displaystyle\lim_{n \to \infty} \frac{\log \left(m^n p_\vartheta(n)\right)}{n}\\
												  		&=& \log (m) + \displaystyle\lim_{n \to \infty}\frac{\log \left(p_\vartheta(n)\right)}{n}\\
												  		&=& \log (m) + h_\vartheta.
\end{array}
\]
\end{proof}
By construction, it is easy to see that $X_{\hat{\vartheta}}$ is the product (as a shift dynamical system) of $X_{\vartheta}$ with $\{1, \ldots, m\}^{\Z}$.
Hence, we could have also used the fact that $h_{\operatorname{top}}(X \times Y) = h_{\operatorname{top}}(X) + h_{\operatorname{top}}(Y)$ to prove Proposition \ref{PROP:log2}.

An important use of Proposition \ref{PROP:log2} is that we can now easily construct RS-subshifts with the same positive entropy but which are not topologically conjugate.
One simply takes a pair of non-conjugate deterministic substitutions $\dsub_1$ and $\dsub_2$ and then forms the new substitutions $\hat{\dsub}_1$ and $\hat{\dsub}_2$ for $m=2$ so that $h_{\hat{\dsub}_1} = h_{\hat{\dsub}_2} = \log (2)$.
There is still some work to show that $X_{\hat{\dsub}_1}$ and $X_{\hat{\dsub}_2}$ are not topologically conjugate---we leave that detail to the reader.
This suggests that, although entropy is a useful and robust invariant for studying the dynamics of RS-subshifts, further tools will be needed in the quest for a dynamical classification.

As illustrated in the next example, we can also produce substitutions on as few as three letters which have arbitrarily small positive topological entropy.
Moreover, because the following example factors onto an aperiodic minimal subshift, the substitution can be chosen to have no periodic points.
\begin{example}
We will show that the substitution $\vartheta \colon a \mapsto \{b^k\}, \bar{a} \mapsto \{b^k\}, b \mapsto \{b^{k-1}a, b^{k-1}\bar{a}\}$ has topological entropy $h_\vartheta \leq \frac{1}{k}\log 2$.
Let $\dsub \colon A \mapsto B^k, B \mapsto B^{k-1}A$ be a primitive deterministic substitution.
Note that the topological entropy of any deterministic substitution is $0$, so $h_\dsub = 0$.
It is clear that $X_\vartheta$ factors onto $X_\dsub$ via the one-block code $f \colon a \mapsto A, \bar{a} \mapsto A, b \mapsto B$.

Every $\dsub$-admitted word of length $k$ has at most one $A$ appearing, and by extension every $\dsub$-admitted word of length $nk$ has at most $n$ appearances of $A$.
It follows that $f \colon \mc L^{nk}(X_\vartheta) \to \mc L^{nk}(X_\dsub)$ is everywhere at most $2^n$-to-$1$ as there are two possible preimages $f^{-1}(u)$ for every appearance of the letter $A$ in the word $u \in \mc L^{nk}(X_\dsub)$.
It follows that $p_\vartheta(nk) \leq 2^n p_\dsub(nk)$.
We can then calculate
\[
\begin{array}{rcl}
h_\vartheta		&	=	& \displaystyle\lim_{n \to \infty} \frac{\log \left(p_\vartheta(nk)\right)}{nk}\\
				&	\leq	& \displaystyle\lim_{n \to \infty} \frac{1}{nk}\log \left(2^n\right) + \lim_{n \to \infty}\frac{\log \left(p_\dsub(nk)\right)}{nk}\\
				&	=	& \frac{1}{k}\log (2) + h_\dsub\\
				&	=	& \frac{1}{k}\log (2).
\end{array}
\]
From results in \cite{RS:random}, we know that $0 < h_\vartheta$ and so we have $0 < h_\vartheta \leq \frac{1}{k}\log 2$.
For any $\epsilon$, we can choose a large enough $k$ so that $h_\vartheta < \epsilon$.
\end{example}
The construction presented in the above example motivates the general construction used in the next proposition.
For any RS-subshift, we want to construct extensions with arbitrarily small increases in entropy.
As a first step, we present a family of subshifts with arbitrarily small upper bounds for the associated entropies.

\begin{lemma}\label{LEM:vanishing-entropy}
Let $(k_n)_{n \in \N}$ and $(K_n)_{n \in \N}$ be sequences of integers satisfying $c_1 k^n \leq k_n \leq K_n \leq c_2 K^n$ for all $n \in \N$ and some fixed $c_1, c_2, k, K \in \R$ with $1 < k \leq K$ and $c_1, c_2 > 0$.
For $ x \in\{ 0,1 \}^{\Z}$ let $(q(x)_m)_{m \in \Z}$ be the sequence of all integers such that $x_{q(x)_m} = 1$ for all $m \in \Z$.
Without loss of generality, we can choose the sequence $q(x)$ to be strictly increasing for all $x \in \{0,1\}^{\Z}$.
Define
\[
X_{(n)} = \{ x \in \{ 0,1 \}^{\Z} \mid k_n \leq q(x)_{m+1} - q(x)_{m} \leq K_n \;\, \text{for all} \;\, m \in \Z \},
\] 
in other words, the subshift of all sequences with uniform lower and upper bounds for the distance between two consecutive appearances of $1$, given by $k_n$ and $K_n$, respectively.
Then,
\[
\lim_{n \rightarrow \infty} h_{\operatorname{top}}\left(X_{(n)}\right) = 0.
\]
\end{lemma}

\begin{proof}
Let $v \in \mathcal{L}(X_n)$ be any legal word of length $N \in \N$.
Then, $v$ contains at most $\left\lfloor\frac{N-1}{k_n}\right\rfloor + 1 $ times the letter $1$ by the definition of $X_{(n)}$.
Thus, $v$ is certainly a prefix of some word $w$ in 
\[
A_N = \left\{ w \in \mathcal{L}\left(X_{(n)}\right) \:\left|\: |w|_1 = \left\lfloor\frac{N-1}{k_n}\right\rfloor + 1 \right. \right\} .
\]
Note that the occurrences of $1$ partition the length of $w$ into $\left\lfloor\frac{N-1}{k_n}\right\rfloor + 2$ intervals, each of them containing between $k_n-1$ and $K_n-1$ occurrences of $0$, except for the boundary intervals which can contain between $0$ and $K_n -1$ occurrences of $0$.
Therefore, the complexity function of $X_{(n)}$ fulfils
\[
p_{(n)}(N) \leq \operatorname{card} \, (A_N) = K_n^2 (K_n - k_n + 1)^{\left\lfloor\frac{N-1}{k_n}\right\rfloor}.
\]
This yields
\begin{align*}
h_{\operatorname{top}}\left(X_{(n)}\right) &= \lim_{N \rightarrow \infty} \frac{1}{N} \log\left(p_{(n)}(N)\right) \leq \lim_{N \rightarrow \infty} \frac{\left\lfloor\frac{N-1}{k_n}\right\rfloor}{N} \log\left(K_n^2 (K_n - k_n + 1)\right)
\\ & =  \frac{1}{k_n} \log\left(K_n^2 (K_n - k_n + 1)\right) \leq C \frac{n}{k^n} \log(K) \xrightarrow{n \rightarrow \infty} 0,
\end{align*}
for some constant $C > 0$.
\end{proof}

\begin{remark}
As an aside, let us mention that the subshift $X_{(n)}$ defined in Lemma~\ref{LEM:vanishing-entropy} is in fact an SFT. It is defined by the finite set of forbidden words 
\[
\mc F_n = \{11, 101, \ldots, 1\overbrace{00\cdots00}^{k_n-2} 1\}\cup \{\overbrace{00\cdots00}^{K_n}\}.
\]
\end{remark}

\begin{definition}
A random substitution $\vartheta$ is of minimum exponential growth if there is a growth rate $r > 1$ and a constant $C>0$ such that 
\[
C r^n \leq \min \, \{ |w| \mid w \overset{\bullet}{=} \vartheta^n(a), a \in \mathcal{A} \}
\]
holds for all $n \in \N$.
\end{definition}
\begin{remark}
An equivalent criterion for $\vartheta$ to be of minimum exponential growth is that there is a $k \in \mathbb{N}$ such that $\vartheta^k$ is certainly growing.
That is,	
\[
| w | \geq 2 \quad \text{for all} \quad w \blacktriangleleft \vartheta^k(a),\; a \in \mc A. 
\]

Note that it is enough to check this condition for $k = m = \operatorname{card} \left( \mc A \right)$ because any sequence of letters $(b_j)_{j \in \mathbb{N}}$ with $b_{j+1} \blacktriangleleft\vartheta(b_j)$ would contain at least one letter twice within the first $m + 1$ elements.
This is closely related to the study of multitype branching processes where the above condition is known to characterise what is known as the \emph{B\"ottcher case} \cite{J:galton-watson}.

\end{remark}

\begin{prop} \label{PROP:log-epsilon}
Let $\epsilon >0$ and let $\vartheta$ be a non-deterministic primitive random substitution of minimum exponential growth on an alphabet $\mc A$ with RS-subshift $X_{\vartheta}$.
Then there exists a primitive random substitution $\hat{\vartheta}$ and a factor map $f \colon X_{\hat{\vartheta}} \to X_\vartheta$ such that
\[
h_\vartheta < h_{\hat{\vartheta}} < h_\vartheta + \epsilon.
\]
\end{prop}

\begin{proof}
For any $\epsilon > 0$, we will give an explicit construction of a primitive random substitution $\hat{\vartheta}$ satisfying the above condition.
Note that $\vartheta$ certainly has a maximum exponential growth rate such that 
\[
\max \{ |w| \mid w \overset{\bullet}{=} \vartheta^n(a), a \in \mathcal{A} \} \leq R^n, \:\: \text{for all} \; n \in \N.
\]
In particular we could choose $R = \max_{a \in \mc A} \max \{|w| \mid w \overset{\bullet}{=} \vartheta(a)\}$. Let $m = |\mc A|$.
By assumption, there exists $n_0 \in \N$ with $|w| \geq m$ for all $w \overset{\bullet}{=} \vartheta^{n_0}(a)$ and any $a \in \mc A$.
Because of primitivity there is a $p \in \N$ such that for $n \geq p$ we have $a \blacktriangleleft \vartheta^{n-n_0}(b)$ for all $a,b \in \mc A$.
Then, every realisation $w^a \overset{\bullet}{=} \vartheta^n(a)$ can be decomposed as $w^a = w^a_{(0)} \cdots w^a_{(\ell-1)}$ with $w^a_{(j)} \overset{\bullet}{=} \vartheta^{n-n_0}(\vartheta^{n_0}(a)_j)$ for all $j \in \{0, \ldots, \ell -1 \}$ and some $\ell \geq m$.
We can choose a realisation $\widetilde{w}^a_{(0)}$ containing each letter of the alphabet.
Since we assume that $\vartheta$ is genuinely a random substitution, we can choose $n$ large enough that $\vartheta^{n-n_0}(b)$ contains at least two distinct realisations for each $b \in \mc A$.
In particular, there are at least $2^{m-1}\geq m$ distinct extensions $\widetilde{w}^a$ of $\widetilde{w}^a_{(0)}$ (in the notation above).
We can therefore choose realisations $\widetilde{w}^a$ with $\widetilde{w}^a_{q_a} = a$ for some position $q_a$ and $\widetilde{w}^a \neq \widetilde{w}^b$ as long as $a \neq b$ for all $a,b \in \mc A$. 

Next we want to define a substitution $\hat{\vartheta}_n$ on the extended alphabet $\hat{\mc A} = \mc A \times \{0,1\}$ by introducing an additional degree of freedom on \emph{each} of the possible realisations of $\vartheta^n(a)$, $a \in \mc A$.
To this end, we first introduce for $i \in \{0,1\}$ the canonical lift $\varphi_i\colon \mathcal{A} \rightarrow \hat{\mc A}$ by $\varphi_i(a) = a_i$ with obvious extension to a homomorphism on finite and infinite words.
Let $\phi$ and $\phi'$ be the injective maps on the set $\cup_{a \in \mc A} \vartheta^n(a)$, defined via $\phi(w^a)_k = \varphi_0(w^a_k)$ and
\begin{align*}
\phi'(\widetilde{w}^a)_k &= \begin{cases}
\varphi_1\left(\widetilde{w}^a_k\right), \quad \text{for} \; k = q_a \\
\varphi_0(\widetilde{w}^a_k), \quad \text{for} \; k \neq q_a
\end{cases} \\
\phi'(w^a)_k &= \begin{cases}
\varphi_1(w^a_k), \quad \text{for} \; k = 0 \\
\varphi_0(w^a_k), \quad \text{for} \; k \neq 0
\end{cases} \quad \text{if} \; w^a \neq \widetilde{w}^b \quad \text{for all} \; b \in \mc A,
\end{align*}
where $k$ ranges from $0$ to $| w^a | -1$.
That is, $\phi$ attaches an index $0$ to each letter and $\phi'$ does the same except for one position in the word, where it attaches an index $1$.
The special role of $\widetilde{w}^a$ serves to ensure that the letter $a_1$ appears in some of the images of $\phi'$ on $\vartheta^n(a)$.
Note that $\phi'$ is indeed a well-defined map on the \emph{union} of the sets $\vartheta^n(a)$, $a \in \mc A$, since we have made sure that $w^a = w^b$ implies $\phi'(w^a) = \phi'(w^b)$ (even if $a \neq b$) by the above construction.
It is now straightforward to check that the substitution $\hat{\vartheta}_n$ defined by
\[
\hat{\vartheta}_n \colon a_0,a_1 \mapsto \phi(\vartheta^n(a)) \cup \phi'(\vartheta^n(a)) \quad \text{for all} \; a \in \mc A
\]
is a primitive random substitution.
Indeed, $a_0 \blacktriangleleft
 \phi(\vartheta^n(b)) \subset \hat{\vartheta}_n(b_i)$ for any $a,b \in \mc A$, $i \in \{0,1\}$ and $a_1 \blacktriangleleft \hat{\vartheta}_n^2(b_i)$ for all $a,b \in \mc A$, $i \in \{0,1\}$.
Let $f \colon \hat{\mc A}^\Z \to \mc A^\Z$ be given on letters by $f(a_i) = a$ for all $a \in \mc A$ and $i \in \{0,1\}$.
The map $f$ is obviously a factor map and restricts to a factor map on the considered RS-subshifts, as $f(X_{\hat{\vartheta}_n}) = X_{\vartheta^n} = X_{\vartheta}$.

It now remains to show that we can pick $n \geq p$ in a way that $h_\vartheta < h_{\hat{\vartheta}_n} < h_\vartheta + \epsilon$. 
We first show that the upper bound holds for large enough $n \in \N$.
Let $v$ be any legal word of length $N$ in $\mathcal{L}(X_{\hat{\vartheta}_n})$.
Then, $v \overset{\bullet}{=} \hat{\vartheta}_n(w)$ for some legal $w = w_0 \cdots w_{|w| -1}$ and suppose $w$ is minimal with this property.
We can thus write $v = v^0 v^1 \cdots v^{|w|-1}$ with $v^j \overset{\bullet}{=} \hat{\vartheta}_n(w_j)$ for $j \in \{1,\ldots, |w|-2 \}$ as well as $v^0 \blacktriangleleft_s \hat{\vartheta}_n(w_0)$ and $v^{|w|-1} \blacktriangleleft_p \hat{\vartheta}_n(w_{|w|-1})$, where the indices $s$ and $p$ denote inclusion as a suffix or prefix, respectively.
Clearly, the analogous relations hold for the images of $v$ and $w$ under the factor map $f$.
In particular, there exists a decomposition of both $v$ and $w$ into inflation words such that the end points of such words coincide.
Let us set
\[
k_n = \min \, \{ |w| \mid w \overset{\bullet}{=} \vartheta^n(a), a \in \mathcal{A} \}, \quad K_n = \max \, \{ |w| \mid w \overset{\bullet}{=} \vartheta^n(a), a \in \mathcal{A} \}.
\]
Then, any $u \in \mc L^N(X_{\vartheta_n})$ is contained in at most $\left\lfloor\frac{N}{k_n} +2\right\rfloor$ level-$1$ inflation words of $\vartheta_n$.
Identifying every starting position of an inflation word with the letter $1$ and all other positions with $0$, we have at most $p_{(n)}(N)$ possibilities to partition $u$ into inflation words and boundary words where $p_{(n)}$ denotes the complexity function of the shift $X_{(n)}$ as defined in Lemma~\ref{LEM:vanishing-entropy}.
For every inflation word $u' \overset{\bullet}{=} \vartheta^n(a)$, there are exactly two preimages under the factor map that are themselves inflation words.
That is, there are at most $2^{\left\lfloor\frac{N}{k_n} +2\right\rfloor}$ preimages of $u$ preserving a given inflation word structure.
On the other hand each word $v \in \mc L^N (X_{\hat{\vartheta}_n})$ is in the preimage $f^{-1}(u)$ of some $u \in \mc L^N (X_{\vartheta_n})$ for \emph{some} preserved partitioning into inflation words, as we discussed above.
Thereby, 
$
p_{\hat{\vartheta}_n}(N) \leq 2^{\left\lfloor\frac{N}{k_n} +2\right\rfloor} p_{(n)}(N) p_{\vartheta_n}(N),
$
yielding
\[
h_{\hat{\vartheta}_n} = \lim_{N \rightarrow \infty} \frac{1}{N} \log \left(p_{\hat{\vartheta}_n}(N)\right) \leq \frac{1}{k_n} \log(2) + h(X_n) + h_{\vartheta_n}.
\]
Since $k_n \rightarrow \infty$ and $h(X_{(n)}) \rightarrow 0$ for $n \rightarrow \infty$, we can choose an $n$ such that $h_{\hat{\vartheta}_n} < h_{\vartheta_n} + \epsilon$ and the upper bound holds.

It is well-known that if there exists a factor map $X \to Y$ between topological dynamical systems, then $h_{\operatorname{top}}(X) \geq h_{\operatorname{top}}(Y)$.
It is therefore a simple observation that $h_{\hat{\vartheta}_n} \geq h_{\vartheta}$ thanks to the factor map $f$.
The difficulty remains in showing that the inequality is indeed strict.
Let $u \in \mc L^N(X_{\vartheta_n})$, then there is at least one way to partition $u$ into inflation words (and boundary words), all of which have maximal length $K_n$.
That is, $u$ contains at least $\left\lfloor N/K_n\right\rfloor-1$ inflation words.
Thus, there are at least $2^{\left\lfloor N/K_n \right\rfloor-1}$ distinct preimages of $u$ under the factor map $f$.
Consequently, $p_{\hat{\vartheta}_n}(N) \geq 2^{\left\lfloor N/K_n\right\rfloor -1} p_{\vartheta_n}(N)$ and
\[
h_{\hat{\vartheta}_n} \geq  \frac{1}{K_n} \log(2) + h_{\vartheta_n} > h_{\vartheta_n},
\]
for any $n \in \N$.

\end{proof}

\begin{remark}
The statement of Proposition~\ref{PROP:log-epsilon} is still true if we drop the assumption that $\vartheta$ is non-deterministic, since a similar construction as outlined in the proof also works for deterministic substitutions. We leave the details to the interested  reader.
\end{remark}

It would be ideal if the minimum exponential growth assumption on $\vartheta$ could be dropped, however we have yet to find a way to do so. In particular, as currently stated, Proposition \ref{PROP:log-epsilon} can unfortunately not be applied to any of the cycle-substitutions constructed in Section \ref{SEC:main}.

We can conclude from Proposition \ref{PROP:log-epsilon} that $\boldsymbol H$ has no isolated points (at least for those entropy values realised as a random substitution satisfying the hypothesis of the result).
This will soon be superseded by the main result of the section, but we find the above result interesting as it says that, not only can entropy be increased by an arbitrarily small amount, but in such a way that the entropy-increase is realised as an extension of RS-subshifts.
\begin{prop}\label{PROP:log-fractions}
For every $1 \leq \ell \leq k$ and every $m \geq 2$, there exists a primitive random substitution of constant length $\vartheta$ with entropy $h_\vartheta = \frac{\ell}{k}\log (m)$.
\end{prop}
\begin{proof}
Let $\dsub$ be a primitive deterministic substitution of constant length on $\mc A = \{a,b\}$ such that $|\dsub(a)| = \ell$.
Such a substitution exists for all $\ell \geq 2$.
If $\ell=1$, let $\dsub$ be the map $a \mapsto b, b \mapsto a$.
Let $\hat{\mc A} = \mc A \times \{1, \ldots, m\}$ be a new alphabet and define a one-block code $f \colon \hat{\mc A}^\Z \to \mc A^\Z$ by the projection map $f \colon (a,i) \mapsto a, (b,i) \mapsto b$. 

We define a random substitution on $\hat{\mc A}$ by
\[
\vartheta \colon
\begin{cases}
(a,i) \mapsto	& \{u(a,1)^{k-\ell}\mid u \in f^{-1}\dsub(a)\}, \\
(b,i) \mapsto	& \{u(a,1)^{k-\ell}\mid u \in f^{-1}\dsub(b)\},
\end{cases}
\]
for every $i \in \{1, \ldots, m\}$. The substitution $\vartheta$ is primitive by construction (which is why we asked that $\dsub$ is a non-trivial permutation in the case $\ell=1$) and is constant length with length $k$.

Let $\psi$ be a primitive deterministic substitution on $\mc A$ given by $\psi \colon a \mapsto \dsub(a)a^{k-\ell}, b \mapsto \dsub(b)a^{k-\ell}$.
As $\psi$ is deterministic, it has zero entropy and so $h_\psi = 0$.
By construction, the one-block code $f$ provides a factor map from $X_\vartheta$ to $X_\psi$.
The preimage of any word $u \in \mc L^{nk}(X_\psi)$ is a set containing exactly $m^{n\ell}$ elements as $u$ contains exactly $n\ell$ letters whose preimages under $f$ are not determined and those letters whose preimages are not determined have exactly $m$ possible preimages.
It follows that $p_\vartheta(nk) = m^{n\ell}p_\psi(n)$.
We calculate
\[
\begin{array}{rclclcl}
h_\vartheta &=& \displaystyle\lim_{n \to \infty} \frac{\log \left(p_\vartheta(n)\right)}{n}&=& \displaystyle\lim_{n \to \infty} \frac{\log \left(p_\vartheta(nk)\right)}{nk}	&=& \displaystyle\lim_{n \to \infty} \frac{\log \left(m^{n\ell}\right) p_\psi(nk)}{nk}\vspace{10pt}\\
												  		&=& \frac{\ell}{k}\log \left(m\right) + h_\psi&=& \frac{\ell}{k}\log \left(m\right).&&
\end{array}
\]

\end{proof}
\begin{theorem}
The set $\boldsymbol H = \{h_\vartheta \mid \vartheta \mbox{ is a primitive random substitution}\}$ is a dense subset of $\R_{\geq 0}$.
\end{theorem}
\begin{proof}
This is a direct result of Proposition \ref{PROP:log-fractions} in the case $m = 2$, whereby
\[
\log (2)\Q_{\geq 0} \cap [0,\log (2)] = \left\{\frac{\ell}{k}\log (2) \mid 1 \leq \ell \leq k \right\} \subset \textbf{H}.
\]
By repeated applications of Proposition \ref{PROP:log2} for the case $m=2$ we can extend this to the entire positive real line to get
\[
\log (2) \Q_{\geq 0} = \left(\log (2) \Q_{\geq 0} \cap [0,\log (2)]\right) + \log (2) \N = \left\{\left( \frac{\ell}{k} + n \right)\log (2) \mid 1 \leq \ell \leq k, n \in \N \right\} \subset \boldsymbol H.
\]
Noting that a positive linear scaling of $\Q_{\geq 0}$ is a dense subset of $\R_{\geq 0}$ completes the proof.
\end{proof}
Another approach would have been to use Proposition \ref{PROP:log-fractions} over all positive values of $m$ and noting that $\log (m)$ is unbounded, so we can build an increasing sequence of dense subsets of the intervals $[0,\log (m)]$ whose union is then necessarily a dense subset of $\R_{\geq 0}$.
\begin{remark}
There are only countably many different random substitution and so the set $\boldsymbol H$ is certainly a set of zero-measure in $\R_{\geq 0}$.
\end{remark}

It was shown in Section \ref{SEC:main} with Theorem \ref{THM:main} and Corollary \ref{COR:Perron-numbers} that every topologically transitive shift of finite type is a primitive RS-subshift, and so the logarithm of every Perron number is realised as topological entropy for a primitive RS-subshift.
It is likely that not all elements of $\boldsymbol H$ are given as the logarithm of Perron numbers.
A likely candidate for such a value is given by the entropy of the random Fibonacci substitution $\vartheta \colon 0 \mapsto \{01,10\}, 1 \mapsto \{0\}$ whose entropy was heuristically calculated by Godr\`{e}che and Luck to be $h_\vartheta = \sum_{i=2}^\infty \frac{\log (i)}{\tau^{i+2}} \approx 0.444399 \cdots$ where $\tau$ is the golden mean.
This was then proved by Nilsson to be the true entropy \cite{N:fibonacci-entropy}.

\begin{question}
Is every element $x$ of $\boldsymbol H$ given by $x = \log (\alpha)$ for a Perron number $\alpha$? 
\end{question}
%

\section*{Acknowledgements}
The authors wish to thank Michael Baake, Michael Coons and Chrizaldy Neil Manibo for helpful discussions.
This work is supported by the German Research Foundation (DFG) via the Collaborative Research Centre (CRC 1283) and the Research Centre for Mathematical Modelling (RCM${}^2$) through the faculty of Mathematics, Bielefeld University.

\bibliographystyle{jis}
\bibliography{tilings}

\end{document}